\documentclass[12pt,a4paper]{article}
\pdfoutput=1

\usepackage[T1]{fontenc}
\usepackage[utf8]{inputenc}
\usepackage[UKenglish]{isodate}
\cleanlookdateon

\usepackage{amsfonts,amsmath,amssymb,amsthm,mathtools,dsfont,microtype} 
\usepackage[usenames,dvipsnames]{xcolor} 
\usepackage[noBBpl]{mathpazo} 

\usepackage[labelsep=period,labelfont=bf,justification=centering]{caption}
\usepackage{float,graphicx,subcaption}
\usepackage{booktabs,makecell}
\usepackage{enumitem,mdwlist}
\setlist[itemize]{topsep=0ex,itemsep=0ex,parsep=0.4ex}
\setlist[enumerate]{topsep=0ex,itemsep=0ex,parsep=0.4ex}

\usepackage{parskip,fullpage}

\usepackage[hyphens]{url} 
\usepackage[linktoc=all,hidelinks,colorlinks,unicode=true]{hyperref} 
\usepackage[capitalise,compress,nameinlink,noabbrev]{cleveref} 
\hypersetup{linkcolor={blue!70!black},citecolor={green!70!black},urlcolor={blue!70!black}}

\usepackage{tikz}

\newtheorem{theorem}{Theorem}
\newtheorem{lemma}[theorem]{Lemma}

\newtheorem{corollary}[theorem]{Corollary}
\theoremstyle{definition}
\newtheorem{remark}[theorem]{Remark}

\renewcommand{\ge}{\geqslant}
\renewcommand{\le}{\leqslant}
\renewcommand{\emptyset}{\varnothing}

\newcommand{\defn}[1]{\textcolor{Maroon}{\emph{#1}}}
\newcommand*{\eps}{\varepsilon}
\DeclareMathOperator{\Bin}{Bin}
\DeclareMathOperator{\Var}{Var}
\DeclareMathOperator{\Po}{Po}
\DeclareMathOperator{\polylog}{polylog}

\newcommand*{\bE}{\mathbb{E}}
\newcommand*{\bN}{\mathbb{N}}
\newcommand*{\bP}{\mathbb{P}}

\newcommand*{\cI}{\mathcal{I}}
\newcommand*{\cJ}{\mathcal{J}}
\newcommand*{\cR}{\mathcal{R}}
\newcommand*{\cC}{\mathcal{C}}
\newcommand*{\cX}{\mathcal{X}}
\newcommand*{\cY}{\mathcal{Y}}

\newcommand*{\cD}{\mathcal{D}}
\newcommand*{\cM}{\mathcal{M}}

\title{\texorpdfstring{\vspace{-4ex}}{}Reconstruction of shredded random matrices}

\date{\today}

\author{
Paul Balister\footnotemark[2], \quad
Gal Kronenberg\footnotemark[2]\  \thanks{Supported by the Royal Commission for
the Exhibition of 1851.}, \quad
Alex Scott\footnotemark[2]\ \thanks{Supported by EPSRC grant EP/X013642/1.}, \quad
Youri Tamitegama\footnotemark[2]
}

\begin{document}
\maketitle

\renewcommand{\thefootnote}{\fnsymbol{footnote}} 

\footnotetext[2]{Mathematical Institute, University of Oxford, United Kingdom
(\textsf{\{\href{mailto:balister@maths.ox.ac.uk}{balister},\href{mailto:kronenberg@maths.ox.ac.uk}{kronenberg},\href{mailto:scott@maths.ox.ac.uk}{scott},\href{mailto:tamitegama@maths.ox.ac.uk}{tamitegama}\}\allowbreak @maths.ox.ac.\allowbreak uk}).}

\renewcommand{\thefootnote}{\arabic{footnote}} 

\begin{abstract}
A matrix is given in ``shredded'' form if we are presented with the multiset of rows and the multiset of columns,
but not told which row is which or which column is which. The matrix is reconstructible if it is uniquely
determined by this information. Let $M$ be a random binary $n\times n$ matrix,
where each entry independently is $1$ with probability $p=p(n)\le\frac12$. Atamanchuk, Devroye and Vicenzo introduced the
problem and showed that $M$ is reconstructible with high probability for $p\ge (2+\eps)\frac{1}{n}\log n$.
Here we find that the sharp threshold for reconstructibility is at $p\sim\frac{1}{2n}\log n$.
\end{abstract}

\section{Introduction}

Let $M$ be an $n\times n$ matrix with entries all either $0$ or~$1$ and let $\cR$ and $\cC$
be the collections (multisets) of the $n$ binary strings of length $n$ representing
the rows and columns of $M$ respectively. When is it possible to reconstruct $M$ just
from the knowledge of $\cR$ and~$\cC$? 

The matrix $M$ is \defn{reconstructible} (or \defn{weakly reconstructible})
if $M$ is uniquely determined by the multisets $\cR$ and $\cC$ of its rows and columns.
We say $M$ is \defn{strongly reconstructible} if the positions of all rows and columns
are determined by $\cR$ and~$\cC$, that is, for each row $r=(r_1,\dotsc,r_n)\in \cR$
we can determine a unique $i$ such that $M_{i,j}=r_j$ for all~$j$, and similarly for columns.
Clearly a weakly reconstructible matrix $M$ is strongly reconstructible if and only if there are
no two identical rows and no two identical columns, i.e., $\cR$ and $\cC$ are actually sets.

We study the threshold for reconstructibility when entries of $M$ are random, independently
chosen to be $1$ with probability $p\le \frac12$. Note that corresponding results for $p\ge\frac12$
can be obtained by exchanging 0 and~1; so we will always assume $p\le \frac12$.
In \cite{atamanchuk2023algorithm}, Atamanchuk, Devroye and Vicenzo showed that for any
$\eps>0$, $M$ is strongly reconstructible with high probability whenever
$(2+\eps)\frac1n\log n\le p \le \frac{1}{2}$. (We say that an event $E$ happens
\defn{with high probability} (\defn{w.h.p.}) if $\bP(E) \to 1$ as $n\to\infty$.)

As mentioned above, a simple obstruction to strong reconstructibility is when two rows or two columns are equal,
and $p\sim\frac{1}{n}\log n$ is a sharp threshold for such an event (see \Cref{lem:equalrows}).
Here we show that this is the main obstacle, locating the threshold for strong reconstructibility
at $p \sim\frac{1}{n}\log n$. In fact, we show a stronger statement. We prove that the threshold for
(weak) reconstructibility is $p\sim\frac{1}{2n}\log n$, so above this value of~$p$, with high probability,
the only likely obstruction to strong reconstructibility is the presence of duplicate rows and/or columns. 
Moreover, we also identify the main obstacle to weak reconstructibility as a pair of 1s in~$M$,
each of which is the only $1$ in its row and column. The threshold $p\sim\frac{1}{2n}\log n$
for weak reconstructibility is simply the threshold for the disappearance of these obstacles
(see \Cref{lem:two1s}).

Our main result is the following. 

\begin{theorem}\label{thm:main}
 Suppose that\/ $p=\frac{1}{2n}(\log n+\log\log n+c_n)\le \frac12$.
 \begin{itemize}
 \item[(a)] If\/ $c_n\to\infty$ as $n\to\infty$, then with high probability
  $M$ is reconstructible from the collection of its rows and columns.
 \item[(b)] If\/ $c_n\to-\infty$ as $n\to\infty$, and assuming that\/ $M$ has at least two $1$'s,
  then with high probability $M$ is not reconstructible from the collection of its rows and columns.
 \item[(c)] If\/ $c_n\to c$ as $n\to\infty$, then the probability that\/ $M$ is reconstructible tends
  to an explicit constant depending on $c$ that is strictly between $0$ and\/~$1$.
 \end{itemize}
\end{theorem}

The proof of \Cref{thm:main} runs in two phases. 
We begin by knowing the ``row values''~$\cR$ and ``column values''~$\cC$, but not where the rows and columns are placed.
In the first phase, we attempt to
assign most of the rows and columns to the correct ``row position'' or ``column position'' using ``local'' information.
It is helpful here to consider the matrix as the adjacency matrix of a bipartite graph~$G$ with bipartition
$(\cI,\cJ)$, where the vertex class~$\cI$ corresponds to the indices of the rows of~$M$, the vertex class
$\cJ$ corresponds to the indices of the columns of~$M$, and the edges correspond to the 1s in the matrix.
Hence for $i\in\cI$, $j\in\cJ$,
$ij$ is an edge if and only if $M_{i,j}=1$. (We assume $\cI$ and $\cJ$ are disjoint, but both have natural bijections
to~$[n]\coloneqq\{1,2,\dotsc,n\}$.) We deduce information about the local structure of $G$ in two different ways:
\begin{itemize}
 \item For row values $r\in\cR$, we deduce the structure of the ball of radius 3 around $r$ by looking
  at the multiset $\cR$ of rows of $M$, which gives the adjacencies in $G$ between the row values in $\cR$
  and the column indices (positions) in~$\cJ$.
 \item For row positions $i\in\cI$, we deduce the structure of the ball of radius 3 around the $i$th row of $M$
  by looking at the multiset $\cC$ of columns of $M$, which gives the adjacencies between the row
  indices in $\cI$ and the column values in~$\cC$.
\end{itemize}
We complete the first part of the argument by matching up most row values with their row positions
and most column values with their column positions. To do this we show that with high probability
(for $p>\frac{\delta}{n}\log n$ with any $\delta>0$), the ball of radius~3 is enough to uniquely
identify most vertices of~$G$. Here we use a method similar to Johnston, Kronenberg, Roberts,
and Scott~\cite{johnston2023shotgun} for reconstructing Erd\H{o}s-Renyi random graphs.

Once we have identified the correct position for most rows and columns, we show in the second phase that we
now have enough information to fill in the remaining rows and columns with high probability
unless certain substructures occur, and identify the thresholds for these substructures occurring.

In~\cite{atamanchuk2023algorithm}, Atamanchuk, Devroye and Vicenzo also proved that for $p\ge \frac{(16+\eps)\log^2n}{n(\log\log n)^2}$
there is an algorithm that succeeds in producing a strong reconstruction of the matrix in $O(n^2)$ time
with high probability and in expectation.
Our proof gives an algorithm that also produces the matrix in $O(n^2)$ time with high
probability above the threshold for weak reconstructibility. See \Cref{re:alg} for more details.

As mentioned above, in light of \Cref{lem:equalrows}, we obtain the threshold for strong
reconstructibility as a simple corollary of \Cref{thm:main}.

\begin{corollary}\label{cor:StrongRec}
 Suppose that\/ $p=\frac{1}{n}(\log n+c_n)\le \frac12$.
 \begin{itemize}
  \item[(a)] If\/ $c_n\to\infty$ as $n\to\infty$, then with high probability,
   $M$ is strongly reconstructible from the collection of its rows and columns.   
  \item[(b)] If\/ $c_n\to-\infty$ as $n\to\infty$, then with high probability,
   $M$ is not strongly reconstructible from the collection of its rows and columns.   
  \item[(c)] If\/ $c_n\to c$ as $n\to\infty$, then the probability that\/ $M$ is strongly
   reconstructible tends to an explicit constant depending on $c$ that is strictly between $0$ and\/~$1$.
 \end{itemize}
\end{corollary}

\subsection{Discussion and related results}

Reconstruction of binary matrices is closely connected to graph reconstruction. 

\textbf{Reconstruction of bipartite graphs.}  
We note that every binary matrix can be viewed
as a bipartite graph where one part acts as the rows of the graph, and the other as the columns.
There is an edge $ij$ in this graph if there is 1 at the $(i,j)$ entry of the matrix.
Thus, reconstruction of a binary matrix by the collection of its rows and columns, can be achieved
by the reconstruction of the corresponding balanced bipartite graph from the collection of its $1$-balls,
where the centred vertex is unlabelled, but the other vertices are labelled.
We will use this connection in our proofs. Our main theorem thus also says the following.

\begin{corollary}
 Suppose that\/ $p=\frac{1}{2n}(\log n+\log\log n+c_n)\le \frac12$ and let\/ $G$ be
 a random subgraph of\/ $K_{n,n}$ obtained by keeping each edge independently with probability~$p$.
 \begin{itemize}
 \item[(a)] If\/ $c_n\to\infty$ as $n\to\infty$, then with high probability
  $G$ is reconstructible from the collection of its $1$-balls with unlabelled centres.
 \item[(b)] If\/ $c_n\to-\infty$ as $n\to\infty$, and assuming that\/ $G$ has at least two edges,
  then with high probability $M$ is not reconstructible from the collection of its $1$-balls with unlabelled centres.
 \item[(c)] If\/ $c_n\to c$ as $n\to\infty$, then the probability that\/ $G$ is reconstructible
  from the collection of its $1$-balls with unlabelled centres tends
  to an explicit constant depending on $c$ that is strictly between $0$ and\/~$1$.
 \end{itemize}
\end{corollary}

\textbf{Reconstruction of directed graphs.} 
Every $n\times n$ binary matrix is equivalent to a directed graph on $n$ vertices,
where a directed edge $\vec{ij}$ appears if and only if there is 1 in the $(i,j)$ entry
of the matrix (possibly with self-loops). Thus, reconstruction of a binary matrix by the collection
of its rows and columns, is equivalent to the reconstruction of the a directed graph by the collection
of its in- and out- neighbourhoods, where the centre vertex is unlabelled, but the other vertices
are labelled. As above, it is straightforward to write down a corollary of \cref{thm:main} for random directed graphs. 

\textbf{Reconstruction of (random) graphs.}
There is a huge literature on graph reconstruction, and a growing body of work on
reconstructing random graphs and other random combinatorial structures.
Kelly and Ulam conjectured in 1941 that every finite simple graph on at
least $3$ vertices can be determined (up to isomorphism) by its collection of vertex-deleted subgraphs,
that is the multiset $\{G\setminus\{v\} \colon v\in V(G)\}$ 
of subgraphs of $G$ obtaining by deleting one vertex each time
\cite{kelly1942isometric,ulam1960collection}.
The Reconstruction Conjecture has a long history and was proved in some special cases, but the general statement is still open
(see e.g., \cite{bondy1977graph,bondy1991graph,asciak2010survey,lauri2016topics} for surveys and background). 

M\"uller~\cite{muller1976probabilistic} proved in 1976 that the Reconstruction Conjecture holds for almost all graphs,
in the sense that it holds with high probability for the binomial random graph $G(n,\frac12)$.
Bollob\'as~\cite{bollobas1990almost} subsequently showed a much stronger result:
in fact, with high probability, the graph can be reconstructed from any three of the subgraphs $G\setminus v$. This led to the
understanding that for the reconstruction of random combinatorial objects, much less information is needed.  

A significant line of recent research looks at when graphs can be reconstructed from ``local'' information.  
Mossel and Ross \cite{mossel2017shotgun} introduced  the ``shotgun reconstruction'' problem.
The terminology was motivated by the shotgun assembly problem for DNA sequences, where the goal
is to reconstruct a DNA sequence from random local ``reads'', corresponding to short subsequences 
(see \cite{dyer1994probability,arratia1996poisson,motahari2013information} among many references).
In the context of graphs, the goal is to reconstruct the graph from balls of small radius.
For a graph $G$ and a vertex $v$ of~$G$, let $N_G^r(v)$ be the induced graph of the vertices of
distance at most $r$ from~$v$ (where only $v$ is labelled).
Then $G$ is \defn{$r$-reconstructible} if every graph with the same multiset of
$r$-balls as $G$ is isomorphic to~$G$. In other words, $G$ can be
identified up to isomorphism from the multiset $\{N_G^r(v)\colon v\in V(G)\}$. 

Mossel and Ross \cite{mossel2017shotgun} proved  that if $p=\frac{\lambda}{n}$, for $\lambda>1$,
then $r=\Theta(\log n)$ is enough for $r$-reconstruction of $G(n,p)$ with high probability.
Sharp asymptotics was obtained by Ding, Jiang, and Ma in 2021~\cite{ding2022shotgun}. 
Mossel and Ross also looked at the problem of reconstructing from balls of constant radius.
They showed that if $np/\log^2n\to \infty$, then $r=3$ is enough for $G(n,p)$
with high probability.  This was later improved by Gaudio and Mossel~\cite{gaudio2020shotgun} who also
obtained bounds on $p$ for the cases $r=1$ and~$2$. The case $r=1$ was improved by Huang and
Tikhomorov~\cite{huang2021shotgun}, who showed that there is a phase transition in 1-reconstructibility
around $p=n^{-1/2}$, where the upper and lower bounds differ by a polylogarithmic factor.
In \cite{johnston2023shotgun}, this problem was settled for $r\ge 3$, and improved bounds were obtained
for $r=1$ and~$2$.

There is also work on other graph models including random regular graphs~\cite{mossel2015shotgun},
random geometric graphs \cite{adhikari2022geometric} and random simplicial complexes \cite{adhikari2022shotgun}.
 
Going back to reconstruction of matrices, one may see this study as the complement of the previous one.
Here, the subgraphs given have all vertices labelled except of the centre, while in the Mossel-Ross type
of random graph reconstruction the only labelled vertex in each given subgraph is the middle one.
As we will see later, the fact that in the matrix case the neighbourhood is labelled, will allow us to
obtain information on larger and larger balls, and by that to apply some of the techniques that were
presented in \cite{johnston2023shotgun} for $r$-reconstruction, where $r\ge 3$. In this case, however,
the notion of $r$-neighbourhoods is slightly different, as all vertices at odd distance from the root
of the ball are labelled but all vertices at even distance are not.
But the informative structure of the matrix does allow us to obtain statistics on the
vertices of distance at most $r$ of each vertex, which will be sufficient for reconstruction. 

\subsection{Organisation}

The paper is organised as follows. In \Cref{sec:prelim} we give classical probabilistic results and 
thresholds for the substructures that provide obstacles to either weak or strong reconstructibility.
In \Cref{sec:most} we start the proof of \Cref{thm:main} by showing how to reconstruct almost all rows and columns.
In \Cref{sec:all} we show that we can complete this into a full reconstruction unless some
specific substructures occur, and complete the proof of \Cref{thm:main}. Finally in \Cref{sec:alg}
we show that reconstruction can be achieved w.h.p.\ in $O(n^2)$ time.

\section{Preliminaries}\label{sec:prelim}

In this section we state probabilistic bounds which will be useful later in the paper.

We make frequent use of the following well-known bounds on the tails of the binomial distribution,
known as Chernoff bounds (see e.g.,~\cite{mitzenmacher2017probability}, Theorem~4.4).

\begin{lemma}\label{lem:chernoff}
 Let\/ $0<p\le\frac12$, $X\sim\Bin(n,p)$ and\/ $\eps>0$. Then,
 \begin{align*}
  \bP\big(X \ge (1+\eps)np\big) & \le \exp \Big(-\tfrac{\eps^2 np}{2+\eps} \Big),\\
  \bP\big(X \le (1-\eps)np\big) & \le \exp \Big(-\tfrac{\eps^2 np}{2} \Big).
 \end{align*}
\end{lemma}
We will also be interested in tail bounds for binomial distributions where
$np \to 0$ as $n \to \infty$, for which we use the following simple observation.

\begin{lemma}\label{lem:small-mean}
 Let\/ $X \sim \Bin(n,p)$ and $k\in\bN$. Then
 \[
  \bP{(X \ge k)} \le \tbinom{n}{k}p^k\ .
 \]
\end{lemma}
\begin{proof}
Indeed, $\bP(X\ge k)$ is at most the expected number of $k$-tuples of trials
that all succeed, which is $\binom{n}{k}p^k$.
\end{proof}

We now quote the following result on convergence to a Poisson distribution which follows
directly from \cite[Theorem~1.23]{bollobas2001random}.
Here we use the \defn{falling factorial} notation $(n)_r:=n(n-1)\dotsb(n-r+1)$.

\begin{lemma}\label{lem:poisson} 
 Let\/ $\lambda_1,\dotsc,\lambda_k$ be non-negative reals and\/ $X^{(i)}_n$, $i=1,\dotsc,k$,
 be sequences of random variables such that for all\/
 $k$-tuples $(r_1,\dotsc,r_k)$ of non-negative integers,
 \[
  \bE\big[(X^{(1)}_n)_{r_1}\dotsb (X^{(k)}_n)_{r_k}\big]\to\lambda_1^{r_1}\dotsb\lambda_k^{r_k}
  \qquad\text{as }n\to\infty,
 \]
 Then $(X^{(1)}_n,\dotsc,X^{(k)}_n)$ converges jointly in distribution to
 independent Poisson random variables with means $\lambda_1,\dotsc,\lambda_k$.
 Namely, for any non-negative integers $s_1,\dotsc, s_k$,
 \[
  \lim_{n\to\infty}\bP\big(X_n^{(1)}=s_1,\dotsc, X_n^{(k)}=s_k\big)
  = \prod_{i=1}^k \frac{e^{-\lambda_i} \lambda_i^{s_i}}{s_i!}\ .
 \]
\end{lemma}

Recall that $M$ is an $n\times n$ matrix with i.i.d.\ Bernoulli random entries
that are 1 with probability~$p$.
The following lemma gives the probability for the main obstruction to strong reconstructibility,
which is the appearance of two identical rows or two identical columns.

\begin{lemma}\label{lem:equalrows}
 Assume\/ $p=\frac{1}{n}(\log n + c_n)\le\frac12$. Then, in the matrix~$M$,
 \[
  \bP(\exists\text{ two equal rows or two equal columns})\to
  \begin{cases}1,&\text{if\/ }c_n\to-\infty,\\
   1-\big((1+e^{-c})e^{-e^{-c}}\big)^2,&\text{if\/ }c_n\to c,\\
   0,&\text{if\/ }c_n\to\infty.
  \end{cases}
 \]
\end{lemma}
\begin{proof}
Let $N$ be the number of pairs of identical rows in~$M$ that are \emph{not} entirely zero.
The probability that two given rows are identical is $((1-p)^2+p^2)^n$ and the probability that
the two rows are identically zero is $(1-p)^{2n}$. Hence,
\begin{align*}
 \bE[N] &= \binom{n}{2}\Big(\big((1-p)^2+p^2\big)^n-\big((1-p)^2\big)^n\Big)\\
 &\le n^2\cdot np^2\cdot \big((1-p)^2+p^2\big)^{n-1}\\
 &\le n^2\cdot np^2\cdot e^{-2(n-1)p(1-p)},
\end{align*}
where we have used comparison with a geometric series (or the Mean Value Theorem) in the second line,
and the inequality $1-x\le e^{-x}$ in the third line. This last expression
tends to zero except in the case $p\le (\frac12+o(1))\frac1n\log n$, in which case $c_n\to-\infty$.

Hence it is enough to consider pairs of identically zero rows or columns.
As the number of zero rows and/or zero columns is stochastically decreasing as $p$ increases,
it is enough to prove the result just in the case when $c_n\to c$ as the other two
cases follow from stochastic domination and taking limits $c\to\pm\infty$.

Let $X$ be the number of zero rows and $Y$ the number of zero columns.
Then for $r,s\ge0$, $(X)_r(Y)_s$ counts the number of choices of $r$-tuples of rows and $s$-tuples of columns,
all filled with zero. Such a configuration consists of a union of $r$ rows and $s$ columns with
all $rn+sn-rs$ entries equal to zero. There are $(n)_r$ $r$-tuples of distinct rows
and $(n)_s$ $s$-tuples of distinct columns so, for fixed $r$ and~$s$,
\[
 \bE\big[(X)_r(Y)_s]=(n)_r(n)_s(1-p)^{rn+sn-rs}\to e^{-c(r+s)}\qquad\text{as }n\to\infty,
\]
as $(n)_r/n^r,(n)_s/n^s,(1-p)^{-rs}\to1$ and $n(1-p)^n=\exp(\log n-pn+O(p^2n))\to e^{-c}$ as $n\to\infty$.
Thus by \Cref{lem:poisson}, $(X,Y)$ converges in distribution to i.i.d.\ $\Po(e^{-c})$ random
variables. The probability that there are either two empty rows or two empty columns
is $\bP(X\ge2\text{ or }Y\ge2)$, which converges to the expression given in the statement of the lemma.
\end{proof}
  
Similarly, the following lemma gives the probability of the main obstruction to
weak reconstructibility. We say an entry in $M$ is an \defn{isolated $1$} if the entry
is a $1$ but all other entries in the same row or column are zero.
In the graph $G$ this corresponds to an isolated edge, i.e., a component consisting of a single edge. 
The main obstruction to weak reconstructibility turns out to be the appearance of two or more isolated $1$s
in the matrix. Indeed, in such a case the isolated $1$s must appear in distinct rows and columns,
and any permutation of their rows, say, results in a matrix distinct from~$M$, but
with identical multisets of rows and columns.

\begin{lemma}\label{lem:two1s}
 Suppose\/ $p=\frac{1}{2n}(\log n + \log\log n + c_n)\le\frac12$. Let\/ $X$ be the number of\/ $1$s
 in the matrix $M$ and let\/ $Y$ be the number of isolated\/ $1$s in $M$. Then,
 \[
  \bP(Y\ge2\text{ or }X<2)\to
  \begin{cases}1,&\text{if\/ }c_n\to-\infty,\\
   1-(1+e^{-c}/2)e^{-e^{-c}/2},&\text{if\/ }c_n\to c,\\
   0,&\text{if\/ }c_n\to\infty.
  \end{cases}
 \]
\end{lemma}
\begin{proof}
Note that $X\sim \Bin(n^2,p)$, so the $X<2$ condition is only significant
if $n^2p$ is bounded, i.e., only in the $c_n\to-\infty$ case.

Now $(Y)_r$ counts the number of $r$-tuples
of isolated~$1$s, all of which must lie in distinct rows and columns. Hence,
\[
 \bE\big[(Y)_r\big]=(n)_r (n)_r p^r(1-p)^{2nr-r^2-r}.
\]
If $c_n\to\infty$ then when $p\le n^{-1/2}$,
\[
 \bE[Y]=n^2p(1-p)^{2n-2}=n^2p e^{-2np-O(np^2+p)}=\frac{np}{\log n}e^{-c_n+O(1)}\to 0,
\]
and clearly $\bE[Y]=O(n^2e^{-2np})\to 0$ for larger~$p$.
Hence $\bP(Y\ge 2)\to0$ by Markov and, as noted above, $\bP(X<2)\to 0$ as well.

If $c_n\to c$ then,
\[
 n^2p(1-p)^{2n}=n^2p e^{-2np - O(np^2)}=\frac{np}{\log n}e^{-c_n+o(1)}\to e^{-c}/2.
\]
As $r$ is fixed and $p\to0$, we then have $\bE[(Y)_r]\to (e^{-c}/2)^r$.
Thus $Y$ tends in distribution to a $\Po(e^{-c}/2)$ random variable.
Also $\bP(X<2)\to0$, so $\bP(Y\ge2\text{ or }X<2)$ converges to the expression given.

If $c_n\to-\infty$ but $n^2p\to\infty$ then $\bE[Y]\to\infty$. However,
\[
 \frac{\bE[Y(Y-1)]}{\bE[Y]^2}=\frac{(n-1)^2}{n^2}(1-p)^{-2}\to 1.
\]
Thus $\Var[Y]=\bE[Y(Y-1)]+\bE[Y]-\bE[Y]^2=o(\bE[Y]^2)$.
Hence by Chebychev's inequality (i.e., the second moment method),
\[
 \bP(Y<2)\le \bP\big(|Y-\bE[Y]|>\bE[Y]-2\big)\le \frac{\Var[Y]}{(\bE[Y]-2)^2}\to 0,
\]
so $\bP(Y\ge2)\to 1$.

Finally we may assume $n^2p=O(1)$,
in which case $Y=X$ w.h.p.\ as the probability of any row or column containing at least
two 1s is $O(n^3p^2)=o(1)$. Hence in this case $\bP(Y\ge2\text{ or }X<2)\ge\bP(X=Y)\to 1$.
\end{proof}

The following lemma will allow us to bound the probability that two multisets of i.i.d.\ random variables
are the same. We shall use the notation $[x_1,\dotsc,x_d]$ to denote the multiset
consisting of the elements $x_1,\dotsc,x_d$.

\begin{lemma}\label{lem:maxprobmulti}
 Let $X_1,\dotsc,X_d$ be i.i.d.\ discrete random variables with $\bP(X_i=x)\le p_0$ for all\/~$x$.
 Then for any multiset\/ $\cM$ of\/ $d$ possible values of\/ $X_i$,
 \[
  \bP\big([X_1,\dotsc,X_d]=\cM\big)\le\frac{(2\pi d+2)^{1/2}\ }{(2\pi p_0d+1)^{1/(2p_0)}}
  =O\big(\sqrt{d}(2\pi p_0d)^{-1/(2p_0)}\big).
 \]
\end{lemma}
We note that if $p_0d\le1$ then we have a better simple bound of $d!p_0^d$ obtained by summing over all
permutations $\sigma\in S_d$ the probability $\bP(X_1=x_{\sigma(1)},\dotsc,X_d=x_{\sigma(d)})$,
where $\cM=[x_1,\dotsc,x_d]$. 
\begin{proof}
Let the multiset $\cM$ with the highest probability have $d_i$ copies of an element $x_i$
where $x_i$ occurs with probability $p_i$, $i=1,2,\dotsc$.
We use the following version of Stirling's formula which holds for all $d\ge0$,
\[
 (d/e)^d\sqrt{2\pi d+1}\le d!\le (d/e)^d\sqrt{2\pi d+2}.
\]
We note that this clearly holds for $d=0$ (with the interpretation that $0^0=1$)
and follows easily for $d\ge1$ from the explicit bounds
\[
 (d/e)^d\sqrt{2\pi d}\,e^{1/(12d+1)}\le d!\le (d/e)^d\sqrt{2\pi d}\,e^{1/12d}
\]
proved by Robbins~\cite{robbins}.
Thus,
\[
 \bP\big([X_1,\dotsc,X_d]=\cM\big)=\frac{d!}{d_1!\dotsb d_n!}p_1^{d_1}\dotsb p_n^{d_n}
 \le \sqrt{2\pi d+2}\cdot\prod_{i=1}^n \Big(\frac{dp_i}{d_i}\Big)^{d_i}\frac{1}{\sqrt{2\pi d_i+1}}.
\]
where, without loss of generality, $d_1,\dotsc,d_n>0$ and $d_i=0$ for $i>n$.
Now set $\alpha_i=p_i/p_0$ and $d'_i=d_i/\alpha_i$. We note that
 $dp_i/d_i=dp_0/d'_i$ and $2\pi d_i+1\ge (2\pi d'_i+1)^{\alpha_i}$ as $0<\alpha_i\le 1$.
 Hence we can rewrite the bound as
 \[
   \log\bP\big([X_1,\dotsc,X_d]=\cM\big)\le\log \sqrt{2\pi d+2}
   +\sum_i \alpha_i\Big(d'_i\log \tfrac{dp_0}{d'_i}-\tfrac12\log\big(2\pi d'_i+1\big)\Big).
 \]
 Now maximise over the $d'_i$, assumed just to be non-negative
 reals with $\sum\alpha_id'_i=d$ (and include $d'_i$ with $i>n$ here as well). It is easy to check that
 $f(x)=x\log\frac{dp_0}{x}-\frac12\log(2\pi x+1)$ is concave, where we set $f(0)=0$.
 Indeed $f(x)\to0$ as $x\to0^+$ and 
 $f''(x)=-\frac{1}{x}+\frac{2\pi^2}{(2\pi x+1)^2}=-\frac{(2\pi x-1)^2 +(4-\pi)(2\pi x)}{x(2\pi x+1)^2}<0$
 for all $x>0$. Hence to maximise $\sum \alpha_i f(d'_i)$ subject to $\sum\alpha_i d'_i=d$
 requires taking all $d'_i$ to be equal, say $d'_i=d'$. We do this for all $i$, even $i>n$.
 But then $\frac{d}{d'}=\sum\alpha_i=\sum \frac{p_i}{p_0}=\frac{1}{p_0}$, so all $d'_i=d'=p_0d$. Thus,
 \[
   \log\bP\big([X_1,\dotsc,X_d]=\cM\big)\le\log \sqrt{2\pi d+2}-\tfrac{1}{2p_0}\log\big(2\pi p_0d+1\big),
 \]
 as required. 
\end{proof}

\section{Reconstructing almost all rows and columns}\label{sec:most}

For this section, let $M$ be a binary matrix whose entries are i.i.d.\ Bernoulli random variables
taking value $1$ with probability~$p$, where $\frac{\delta}{n}\log n\le p\le \frac12$
for some fixed small constant $\delta>0$.

Recall that the matrix $M$ corresponds to a bipartite graph $G$ with vertex classes $\cI$
and $\cJ$, both of size~$n$, corresponding to the indices of the rows and columns.
For $i\in\cI$, $j\in\cJ$, $ij$ is an edge of $G$ if and only if the matrix $M$ has 1 in the entry $(i,j)$.

Given the multisets $\cR$ and $\cC$, we can reconstruct two graphs, both isomorphic to~$G$.
The first is $G_R$, which is a bipartite graph with vertex classes $\cR$ and $\cJ$ where
the row value $r\in\cR$, which is a binary vector $r=(r_1,\dotsc,r_n)$, is joined to
all columns $j\in \cJ$ where $r_j=1$. The second is $G_C$, which is a bipartite graph
with vertex classes $\cI$ and $\cC$ where the column value $c=(c_1,\dotsc,c_n)^T\in \cC$
is joined to all rows $i\in \cI$ where $c_i=1$. Clearly $G_R$ and $G_C$ are both
isomorphic to $G$ by the correct identification of the row values in $\cR$ with
their indices in $\cI$, and the column values in $\cC$ with their indices in $\cJ$ respectively.

For a vertex $v$ of $G$ (or $G_R$ or $G_C$), define its \defn{$k$th degree statistics} inductively as follows:
\begin{align*}
 \cD_0(v) &= \deg(v),\\
 \cD_{k+1}(v) &= [\cD_k(u) \colon u\in N(v)], \text{ for } k>0.
\end{align*}
Note that as $G_R$ and $G_C$ are isomorphic to $G$, $\cD_k(v)$ can be reconstructed from
$\cR$, $\cC$, and \emph{either} the index \emph{or} the value of the row or column~$v$. In particular,
we observe that the index of a row value $r$, say, can be correctly identified
if for all rows $r'\ne r$ there is some $k$ such that $\cD_k(r)\ne \cD_k(r')$.

We first show that for large $p$, w.h.p.\ even $\cD_1(v)$ is enough to uniquely identify all rows and columns.

\begin{lemma}\label{lem:largep}
 There exists a $C>0$ such that for $\frac{C}{n}\log^2 n\le p\le \frac{1}{2}$ and any two distinct
 rows $r$ and\/ $r'$ of\/ $M$,
 \[
  \bP\big(\cD_1(r)=\cD_1(r')\big)=o(n^{-2}).
 \]
 In particular, w.h.p.\ $M$ is strongly reconstructible.
\end{lemma}
\begin{proof}
Let $r,r'\in\cR$. If $\deg(r)\ne \deg(r')$ then clearly $\cD_1(r)\ne \cD_1(r')$, so we may assume $\deg(r)=\deg(r')$.
If $\cD_1(r)=\cD_1(r')$ then the multiset $\cM_1$ of degrees of vertices (columns) in $N(r)\setminus N(r')$
is the same as the multiset $\cM_2$ of degrees of vertices in $N(r')\setminus N(r)$.
The multiset of degrees of vertices in $N(r)\cap N(r')$ contributes equally to $\cD_1(r)$ and $\cD_1(r')$,
so we can ignore these.

If we condition on $N(r)$ and $N(r')$ (i.e., the entries in rows $r$ and~$r'$)
and write $d=|N(r)\setminus N(r')|=|N(r')\setminus N(r)|$ then $\cM_1$ and $\cM_2$ are
i.i.d.\ random multisets of size~$d$, each of which consists of
i.i.d.\ $\Bin(n-2,p)+1$ random variables. The $+1$ is because we have exactly
one of $r$ or $r'$ counted in the degrees. The elements of the multisets and the multisets
themselves are independent as they depend only on distinct columns, respectively
disjoint rectangles $(\cI\setminus\{r,r'\})\times(N(r)\setminus N(r'))$
and $(\cI\setminus\{r,r'\})\times(N(r')\setminus N(r))$, of unconditioned entries in~$M$.

Now if $X\sim\Bin(n-2,p)$ then
$p_0\coloneqq\max_x\bP(X=x)=\Theta(1/\sqrt{np})$, which we may assume is at most $1/(4\log n)$ if $C$
is large enough. Hence, by \Cref{lem:maxprobmulti},
\[
 \bP(\cM_1=\cM_2)=O(\sqrt{d})\exp\big(-\tfrac{1}{2p_0}\log(2\pi p_0d)\big)
 =O\big(n^{1-2\log(2\pi p_0d)}\big)=o(n^{-2}),
\]
provided $p_0d > 2$, say. But $d\sim \Bin(n,p(1-p))$ stochastically dominates $\Bin(n,p/2)$
and thus by \Cref{lem:chernoff},
\[
 \bP(d<np/4)\le \exp(-np/16)=o(n^{-2}).
\]
But $d\ge np/4$ implies $p_0d=\Theta(\sqrt{np})>2$ for large~$n$. Hence unconditionally,
\[
 \bP\big(\cD_1(r)=\cD_1(r')\big)=\bP\big(\cM_1=\cM_2\big)=o(n^{-2}).
\]
The union bound now shows that the probability that there are two rows with the same value of
$\cD_1(r)$ is $o(1)$, and the same also holds for columns. Since the $\cD_1(v)$
can be determined either from the values or the indices of the rows or columns~$v$,
w.h.p.\ each row or column value can be associated with a unique index, and so
$M$ is strongly reconstructible.
\end{proof}

For smaller values of $p$ we will need to consider $\cD_2(v)$. Before we do so, it
will be convenient to prove some results about the typical structure of $M$ when
$p=O(\frac1n(\log n)^2)$.

\begin{lemma}\label{obs:atmost2common1s}
 Let $p=O(\frac1n\log^2 n)$.  Then with high probability, every pair of distinct
 rows $r$ and\/ $r'$ satisfy $|N(r)\cap N(r')|\le 2$.
\end{lemma}
\begin{proof}
We have $|N(r)\cap N(r')|\sim\Bin(n,p^2)$, so
\[
 \bP\big(|N(r)\cap N(r')|\ge 3\big) \le \tbinom{n}{3} p^6 = o(n^{-2}).
\]
A union bound now shows that w.h.p.\ there is no pair of rows $r$, $r'$
for which $|N(r)\cap N(r')|\ge 3$.
\end{proof}

The next fact roughly states that the $2$-balls in our bipartite graph are essentially cycle-free.
We will use this to argue that the degree distributions of leaves in these balls are essentially independent.

\begin{lemma}\label{obs:secondneighbourhood}
 For $p=O(\frac1n\log^2n)$ w.h.p.\ there does not exist a pair $r$, $r'$ of rows for which there
 are more than two rows $s\ne r,r'$ with at least two neighbours in $N(r)\cup N(r')$.
\end{lemma}
\begin{proof}
Consider a pair of rows $r$ and $r'$. If rows $s_1,s_2,s_3\ne r,r'$ all have at least 
two neighbours in $N(r)\cup N(r')$, then choose two such neighbours for each of these rows to form
a subset $C\subseteq N(r)\cup N(r')$ of columns of size $k:=|C|\le 6$ with at least 6 edges from $\{s_1,s_2,s_3\}$ to~$C$.
Adding $r$ and~$r'$, it follows that there is a set of $5$ rows and $k$ columns with $k+6$ entries equal to~$1$.
The expected number of these configurations is $O(n^{k+5}p^{k+6})=o(1)$.
Summing over $k\le6$, we see that with high probability there is no configuration of this type,
and so no choice of $r$, $r'$ and the~$s_i$.
\end{proof}

A standard Chernoff bound argument shows the following bound on the maximum number of 1s in any row or column.

\begin{lemma}\label{obs:maxdeg}
 Fix $\delta>0$. Then there exists a constant $K=K(\delta)$ such that
 for $p\ge \frac{\delta}{n}\log n$, w.h.p., no row or column has degree more than $Knp$ in $G$.
\end{lemma}
\begin{proof}
Set $\eps=K-1$ in \cref{lem:chernoff} and note that 
$\frac{\eps^2 np}{2+\eps}\ge \frac{(K-1)^2 \delta}{K+1}\log n\ge 2\log n$
for sufficiently large~$K$. Thus the probability that a fixed row or column has degree
more than $Knp$ is at most $e^{-2\log n}=1/n^2$.
The result follows from the union bound over the $2n$ rows and columns.
\end{proof}

Call a row or column \defn{heavy} if it has at least $\frac{1}{2}np$ ones, and \defn{light} otherwise.

\begin{lemma}\label{obs:heavyrows}\label{lem:few1sinlight}
 Fix $\delta>0$ and assume $p\ge \frac{\delta}{n}\log n$.
 Then w.h.p.\ there are at most $n^{1-\delta/9}=o(n)$ light rows.
 Moreover, w.h.p.\ each row $r$ is adjacent to at most $K'=K'(\delta)$ light columns.
\end{lemma}
\begin{proof}
By \cref{lem:chernoff}, each row is light with probability at most
$e^{-np/8}\le n^{-\delta/8}$, so by Markov there are w.h.p.\ at most $n^{1-\delta/9}$
light rows.

Now given a column $c$ and a row index $i_0\in\cI$, \cref{lem:chernoff} gives
\[
\bP\Big( \sum_{i\neq i_0}c_i < \tfrac{1}{2}np-1\Big)
\le \bP\Big( \sum_{i\neq i_0}c_i < \tfrac{1}{2}(n-1)p\Big) \le e^{-(n-1)p/8},
\]
as $\sum_{i\ne i_0}c_i\sim\Bin(n-1,p)$.
Therefore, given a row $r$, we have that, conditioned on $N(r)$,
the probability that a fixed column $c\in N(r)$ is light is at most $e^{-(n-1)p/8}$.
Hence, as distinct columns are independent,
\[
 \bP\big(\text{There are at least $K'$ light columns in }N(r)\big) \le
 \binom{|N(r)|}{K'}\big(e^{-(n-1)p/8}\big)^{K'}.
\]
Now w.h.p.\ every row $r$ satisfies $|N(r)|\le Knp$ by \cref{obs:maxdeg},
so for $K'\coloneqq\lceil 9/\delta\rceil$, say, we have
\[
 \binom{|N(r)|}{K'}\big(e^{-(n-1)p/8}\big)^{K'}\le e^{K'(\log(Knp)-(n-1)p/8)}=o(n^{-1}).
\]
Taking a union bound over the rows of $M$ finishes the argument.
\end{proof}

Thus, it is enough to reconstruct only heavy rows and columns in order to reconstruct $(1-o(1))$ of the matrix.

\begin{lemma}\label{lem:2nddegreestatistics}
 Fix $C,\delta>0$ and suppose $\frac{\delta}{n}\log n\le p=p(n)\le \frac{C}{n}\log^2 n$.
 Then w.h.p.\ for any pair of distinct rows $r$ and\/~$r'$ with $r$ heavy, $\cD_2(r)\ne\cD_2(r')$.
 In particular, w.h.p.\ all heavy row and column values can be matched with their indices.
\end{lemma}
\begin{proof}
We may assume $\deg(r)=\deg(r')$ as otherwise $\cD_2(r) \ne \cD_2(r')$.
We may also assume the conclusions of Lemmas~\ref{obs:atmost2common1s}--\ref{lem:few1sinlight}
all hold. Reveal all entries in columns of $N(r)\cup N(r')$ and in rows of $N(N(r'))$,
so that the full information determining $\cD_2(r')$ is revealed.
For the second degree statistics to coincide, there needs to be a matching
$\sigma \colon N(r)\to N(r')$ such that $\cD_1(c) = \cD_1(\sigma(c))$ for all columns $c\in N(r)$.
We take a union bound over all $d!$ such matchings where $d=|N(r)|=|N(r')|$.
\begin{align*}
 \bP\big(\cD_2(r) = \cD_2(r')\big)
 &\le \sum_\sigma \bP\big(\forall c\in N(r),\ \cD_1(c) = \cD_1(\sigma(c))\big)\\
 &\le d!\cdot \max_\sigma \bP\big(\forall c\in N(r),\ \cD_1(c) = \cD_1(\sigma(c))\big).
\end{align*}
It now suffices to bound the last probability for any matching~$\sigma$.

We first show that essentially all degree distributions $\cD_1(c)$ are independent
from each other and crucially from the information revealed so far. 
The probability of each $\cD_1(c)$ hitting its prescribed degree distribution is
bounded by \cref{lem:maxprobmulti}. Taking a product over the valid $c$ yields the desired bound.

Say that a column $c\in N(r)$ is \defn{admissible} if none of the following occurs:
\begin{itemize}
\item[(i)] $c\in N(r')$,
\item[(ii)] there is a $c'\in N(r)\cup N(r')$ with $c'\ne c$
 such that $(N(c')\cap N(c))\setminus\{r,r'\}\ne\emptyset$,
\item[(iii)] $\deg(c) < \frac12np$.
\end{itemize}
By \cref{obs:atmost2common1s}, $N(r)$ and $N(r')$ share at most two elements,
so at most two columns $c$ fail because of condition~(i).
By \cref{obs:secondneighbourhood} there are at most two rows $s\ne r,r'$ whose neighbourhood
intersects $N(r)\cup N(r')$ in at least two columns. Since these intersections have size
at most $4$ each, there are at most $8$ columns $c\in N(r)\cup N(r')$ 
for which there is a $c'\in N(r)\cup N(r')$, $c'\ne c$, with $N(c)\cap N(c')$ containing
such a row~$s$. In other words, there are at most 8 columns which fail the condition~(ii).
Lastly, by \cref{lem:few1sinlight}, only constantly many $c\in N(r)$ fail condition~(iii).

As $r$ is assumed heavy, for large $n$ there are at least $\frac{1}{3}np$ admissible vertices in $N(r)$.

Notice that if a column $c$ is admissible, then, after omitting $r$, all rows in its
neighbourhood $N(c)\setminus \{r\}$ have had exactly one non-zero entry revealed so far.
In particular, the degree of each such row only depends on its $N\coloneqq n-|N(r)\cup N(r')|$ entries
yet to be revealed, so there are i.i.d.\ random variables $X_1,\dotsc,X_{\deg(c)}\sim \Bin(N, p)$
such that $\cD_1(c) =[\deg(r), X_1+1,\dotsc, X_{\deg(c)}+1]$.
Additionally, the second admissibility condition ensures that neighbourhoods of distinct
admissible columns are disjoint. Namely, the degrees of the neighbours of distinct admissible
columns $c$, $c'$ depend on disjoint subsets of unconditioned entries and are therefore independent.

Now for any admissible $c$ and choosing any fixed choice of $\cD_1(\sigma(c))$ we have that
\[
 \bP(\cD_1(c)=\cD_1(\sigma(c)))=O(\deg(c)^{1/2})(2\pi p_0\deg(c))^{-1/p_0},
\]
where $p_0=\Theta(1/\sqrt{Np})$. However, $|N(r)\cup N(r')|\le 2Knp$,
so $p_0=\Theta(1/\sqrt{np})$ and by assumption $\deg(c)\ge \frac12np$. Hence,
\[
 \bP(\cD_1(c)=\cD_1(\sigma(c)))\le\exp\big(-\Theta(\sqrt{np}\log(np))\big).
\]
Now, if $\cD_1(c) =\cD_1(\sigma(c))$ holds for every $c\in N(r)$ then certainly it holds
for every admissible $c$, so,
\begin{align*}
 \bP\big( \forall c\in N(r),\ \cD_1(c) = \cD_1(\sigma(c))\big)
 &\le \bP\big( \forall c \text{ admissible},\ \cD_1(c) = \cD_1(\sigma(c))\big)\\
 &= \prod_{c\text{ admissible}} \bP\big(\cD_1(c) = \cD_1(\sigma(c))\big)\\
 &= \exp\big(-\Omega((np)^{3/2}\log(np))\big),
\end{align*}
where we used that there are at least $np/3$ admissible columns.
Finally, note that $d!\le (Knp)! = \exp\big(O(np\log(np))\big)$, so that
\begin{align*}
 \bP\big(\cD_2(r)=\cD_2(r')\big)
 &\le\exp\big(-\Omega((np)^{3/2}\log(np))+O(np\log(np))\big)\\
 &\le\exp\big(-\Omega(\log^{3/2}n)\big) = o(n^{-2}).
\end{align*}
Taking a union bound over all choices of $r$ and $r'$ then gives the result.
\end{proof}

\section{Full reconstruction}\label{sec:all}

In this section we show that once all but $o(n)$ of the rows and columns have been reconstructed,
then with high probability the remaining entries of $M$ can be deduced unless there is some simple obstruction.
Recall that we say an entry in $M$ is an isolated~$1$ if the entry is a $1$ but all other entries in
the same row or column are zeros.

\begin{lemma}\label{lem:fullrecon}
 Fix $\eps,C>0$ and suppose $(\frac13+\eps)\frac1n\log n\le p\le\frac{C}{n}\log^2 n$. Then
 w.h.p.\ every row and column of\/ $M$ can be determined except possibly for rows and columns
 where there is an isolated\/~$1$.
\end{lemma}
\begin{proof}
We can assume, by \Cref{lem:2nddegreestatistics}, that we have already placed all heavy rows
and columns in their correct positions. Write $\cX$ for the set of rows and $\cY$ for the set
of columns that have been placed.

We note that we can place any row $r$ which has a unique neighbourhood $N(r)\cap\cY$
in the already placed columns as we can determine this set from either the
value or the index of~$r$. Moreover, we may assume that in the $k\times\ell$ submatrix $A$
formed from the unplaced rows and columns that every row and every column contains a~$1$. 
Indeed, if~$A$ contained a zero row, corresponding to the row~$r$ of~$M$, say,
then $N(r)=N(r)\cap\cY$. But for any row
$r'$ with $N(r')\cap\cY=N(r)\cap \cY$, either $|N(r')|>|N(r)|$, in which case the
rows are distinguished by their degrees (and we can determine the degrees of rows in any position from $\cC$), or $N(r')=N(r)$, in which case the row
values of $r$ and $r'$ are identical. Thus the position of any row value $r$
is uniquely identified up to permutation of equal rows.

Now suppose the row $r$ is still unplaced but $N(r)\cap\cY\ne\emptyset$.
Then there must be another row $r'$ with $|N(r)| = |N(r')|$ and $N(r)\cap \cY=N(r')\cap \cY$,
but $N(r)\ne N(r')$. As $|N(r)|=|N(r')|$
there are distinct columns $c$, $c'$ with $c\in N(r)\setminus N(r')$ and $c'\in N(r')\setminus N(r)$.
Clearly $c,c'\notin \cY$ so are also unplaced. We may also choose $c$ and $c'$ so that
$N(c)\cap\cX=N(c')\cap\cX$
as otherwise $r$ and $r'$ could be distinguished by the multisets $[N(c)\cap\cX:c\in N(r)\setminus N(r')]$
and $[N(c)\cap\cX:c\in N(r')\setminus N(r)]$, both of which can be identified from
either the values or positions of $r$ and~$r'$.

We now bound the number of $4$-tuples $(r,r',c,c')$ which could satisfy these conditions.
More specifically we count the number of $4$-tuples $(r,r',c,c')$ satisfying the following
slightly weaker conditions. 
\begin{itemize}
\item[(i)] $c\in N(r)\setminus N(r')$ and $c'\in N(r')\setminus N(r)$.
\item[(ii)] $N(r)\cap N(r')\ne \emptyset$. 
\item[(iii)] $N(r)\cap C=N(r')\cap C$ where $C=\{c''\ne c,c':|N(c'')\setminus\{r,r'\}|\ge\frac12np\}$.
\item[(iv)] $N(c)\cap R=N(c')\cap R$ where $R=\{r''\ne r,r':|N(r'')\cap C|\ge \frac12np\}$.
\end{itemize}
Note that all columns in $C$ and all rows in $R$ are heavy, so assumed already placed.
Fixing $(r,r',c,c')$, we note that by a similar calculation
as in \Cref{obs:heavyrows}, any column $c''\ne c,c'$ lies in $C$ with probability at least
$1-n^{-1/25}$ independently of (i) and (ii). Indeed, even conditioned on $N(r)$ and $N(r')$, for any $c''\ne c,c'$,
\[
 \bP(c''\notin C)=\bP\big(\Bin(n-2,p)<\tfrac12np\big)\le e^{-(1-o(1))np/8}\le n^{-1/25},
\]
independently for each~$c''$. Let $E$ be the event that $|C|<n-2-2n^{24/25}=(1-o(1))n$.
Then as the number of $c''\ne c,c'$ not in $C$ is stochastically dominated by a $\Bin(n-2,n^{-1/25})$
random variable, by \Cref{lem:chernoff},
\[
 \bP(E) \le e^{-(1/3)(n-2)n^{-1/25}}=n^{-\omega(1)}.
\]
Now conditioned on $N(r)$, $N(r')$ and $C$, and assuming $E$ does not hold, 
any row $r''\ne r,r'$ lies in $R$ with probability at least $1-n^{-1/25}$. Indeed,
\[
 \bP(r''\notin R)\le \bP\big(\Bin(|C|,p)<\tfrac12np\big)\le e^{-(1-o(1))np/8}\le n^{-1/25},
\]
as an entry being $1$ in row $r''$ is positively correlated with the condition that its column is in $C$.

Now, given $(r,r',c,c')$, the probability that (i) holds is $p^2(1-p)^2$.
Conditioned on this, (ii) holds with probability at most $np^2$. Conditioned on (i) and (ii),
(iii) holds with probability at most
\[
 \big((1-p)^2+p^2+2p(1-p)n^{-1/25}\big)^{n-2}=e^{-2pn+o(1)}\le n^{-2/3-2\eps+o(1)}.
\]
Conditioned on this the probability that (iv) holds but $E$ does not occur is then at most
\[
 \big((1-p)^2+p^2+2p(1-p)n^{-1/25}\big)^{n-2}=e^{-2pn+o(1)}\le n^{-2/3-2\eps+o(1)}.
\]
Thus the expected number of such $4$-tuples is at most
\[
 n^4\cdot \big(\bP(E)+p^2(1-p)^2\cdot np^2\cdot n^{-2/3-2\eps+o(1)}\cdot n^{-2/3-2\eps+o(1)}\big)
 =n^{-1/3-4\eps+o(1)}=o(1).
\]
Hence we may assume $N(r)\cap\cY=\emptyset$ for all unplaced rows $r$ and similarly
$N(c)\cap\cX=\emptyset$ for all unplaced columns~$c$.

Now consider the graph $G$ restricted to the unplaced rows and columns. The above argument
shows that we can assume this forms a union of components in $G$ of total cardinality at most $o(n)$.
Isolated vertices correspond to zero rows or zero columns. Isolated edges correspond to
isolated $1$s in~$M$. Thus is is enough to show that $G$ contains no components with between $3$
and $o(n)$ vertices. We count the expected number of such components by counting
the number of possible choices of spanning trees for such components. We get that the expected
number of these components is then at most
\begin{align*}
 \sum_{k=3}^{o(n)}\binom{2n}{k}p^{k-1}k^{k-2}(1-p)^{k(n-k)}
 &\le \sum_{k=3}^{o(n)}\big(2ne\cdot e^{-p(n-k)}\big)^k p^{k-1}\\
 &\le \sum_{k=3}^{o(n)} n^{1-(1/3+\eps+o(1)) k}=o(1).\qedhere
\end{align*}
\end{proof}

\begin{remark} 
For $p<\frac1{3n}\log n$ another obstruction to reconstructibility appears, namely
pairs of rows (or columns) each of which contains two 1s which themselves are
the unique $1$ in their column (or row). In graph terms these consist of at least
two components that are isomorphic to a path on $3$ vertices (with the central vertices
both in the same bipartite class). In general more complex tree components on $k$
vertices appear for $p<\frac{1}{kn}\log n$ and if two isomorphic copies of a
tree $T$ exist (with the isomorphism respecting the bipartition of~$G$) then
the matrix $M$ fails to be reconstructible as we can interchange the vertices
in one bipartite class of $T$ with their counterparts in another copy of $T$
without affecting the multisets of rows and columns. This does however affect
the matrix $M$ for $k\ge2$. For $k=1$ pairs
of isolated vertices in the same class correspond to pairs of zero rows or
zero columns which is the main obstacle to \emph{strong} reconstructibility,
but do not prevent reconstructing $M$ as no edges are changed when they are
swapped. It should be noted that isolated tree components are not
the only obstacle to reconstructibilty. For example, for $p<\frac1{4n}\log n$
can have paths on $5$ vertices with only the middle vertex possibly joined
to other vertices. This corresponds to say two rows $r,r'$ and three columns $c,c',c''$
with $N(c)=\{r\}$, $N(c')=\{r'\}$, $N(r)=\{c,c''\}$, $N(r')=\{c',c''\}$.
In this case $r$ and $r'$ can be swapped giving a different matrix with the same
multisets of rows and columns.
\end{remark}

\begin{proof}[Proof of \Cref{thm:main}.]
If $(\frac13+\eps)\frac1n\log n\le p\le \frac12$ then by \Cref{lem:largep} or \Cref{lem:fullrecon} we can
w.h.p.\ reconstruct $M$ up to rows and columns with isolated 1s. If there is at most one
isolated~$1$ then clearly we can reconstruct the whole of~$M$. If there are two or more isolated~$1$s then we can't
reconstruct $M$ as permuting the rows containing these isolated~$1$s will give a different
matrix with the same row and column multisets. Hence the result follows from \Cref{lem:two1s},
with the explicit constant in part~\emph{(c)} being as in \Cref{lem:two1s}.

If $p<(\frac13+\eps)\frac1n\log n$ then again by \Cref{lem:two1s} we have that w.h.p.\ there are
either two isolated $1$s or the matrix has fewer than two $1$s in total.
\end{proof}

\section{A fast algorithm for reconstruction}\label{sec:alg}

\begin{lemma}\label{re:alg}
 There exists an algorithm that w.h.p.\ either identifies all rows and columns,
 or finds a pair of isolated~$1$s and takes $O(n^2)$ time.
\end{lemma}
\begin{proof}
We will restrict our attention to the case when $p=O(\frac1n\log^2 n)$ as for larger $p$
the result follows from~\cite{atamanchuk2023algorithm}.

Constructing a list of all neighbours of every row and column index and every row and column
value in $G_R$ and $G_C$ takes $O(n^2)$ time as we have to scan every vector in $\cR\cup\cC$.
As w.h.p.\ there are only $O(\log^2 n)$ neighbours of any
index or value, constructing $\cD_2(v)$ for every $v$ then takes only $O(n\polylog n)$ time.
Sorting and finding matches between indices and values then again takes $O(n\polylog n)$ time.
Identifying any isolated $1$s also takes $O(n\polylog n)$ time. So unless $p\ge(\frac13+\eps)\frac1n\log n$
by \cref{lem:two1s} we will w.h.p.\ have terminated with either a pair of isolated $1$s which demonstrates non-reconstructibility or a reconstructed matrix with at most one non-zero entry.
We may thus assume $p\ge (\frac13+\eps)\frac1n\log n$ from now on.

The final part of the algorithm relies on calculating the multisets $[N(c)\cap\cX:c\in N(r)]$
for rows $r$ and similarly for columns and identifying rows or columns that are uniquely determined.
Again, calculating these multisets takes only $O(n\polylog n)$ time.
If this fails to identify all rows and columns then as in the proof of \Cref{lem:fullrecon} w.h.p.\ the remaining rows
and columns all contain isolated $1$s, which we can easily check.

Thus overall the algorithm takes $O(n^2)$ time (with most of the time taken up in the
initial scanning of rows and columns to find their neighbours) and correctly identifies
$M$ w.h.p.\ or finds a pair of isolated $1$s showing that reconstruction is impossible.
\end{proof}

\section{Conclusion}

We have shown that there is a sharp threshold for reconstructibility at
$p\sim\frac{1}{2n}\log n$, and a sharp threshold for strong reconstructibility at
$p \sim\frac1n\log n$. These results both assume that we know \emph{all} the rows and columns in $\cR$ and $\cC$.  But what if we are missing some of the rows and columns?  For example, suppose we are given $\cR$, but only a (random) subset of $cn$ elements from $\cC$: for what range of $c$ and $p$ can we reconstruct $M$ with high probability?

It would also be very interesting to investigate the problem when there are errors in our data. For example, suppose every entry in each row from $\cR$ is given incorrectly with probability $q$: when can we reconstruct $M$ with high probability? This seems to be interesting even when $p=1/2$ and $q$ is a small constant.  

A similar question arises when we have missing data for both rows and columns: when can we reconstruct almost all of $M$? And, in a different direction, what happens if the data is corrupted adversarially?

\fontsize{11pt}{12pt}
\selectfont
	
\hypersetup{linkcolor={red!70!black}}
\setlength{\parskip}{2pt plus 0.3ex minus 0.3ex}

\bibliographystyle{Bib.bst}
\bibliography{bib.bib}

\begin{thebibliography}{AMRW96}
\providecommand{\url}[1]{\texttt{#1}}
\providecommand{\urlprefix}{\textsc{url:} }
\expandafter\ifx\csname urlstyle\endcsname\relax
  \providecommand{\doi}[1]{doi:\discretionary{}{}{}#1}\else
  \providecommand{\doi}{doi:\discretionary{}{}{}\begingroup \urlstyle{rm}\Url}\fi

\bibitem[AC22a]{adhikari2022shotgun}
\textsc{Kartick Adhikari} and \textsc{Sukrit Chakraborty} (2022).
\newblock Shotgun assembly of {L}inial-{M}eshulam model.
\newblock \emph{arXiv preprint arXiv:2209.10942} .

\bibitem[AC22b]{adhikari2022geometric}
\textsc{Kartick Adhikari} and \textsc{Sukrit Chakraborty} (2022).
\newblock Shotgun assembly of random geometric graphs.
\newblock \emph{arXiv preprint arXiv:2202.02968} .

\bibitem[ADV23]{atamanchuk2023algorithm}
\textsc{Caelan Atamanchuk}, \textsc{Luc Devroye}, and \textsc{Massimo Vicenzo} (2023).
\newblock An algorithm to recover shredded random matrices.

\bibitem[AFLM10]{asciak2010survey}
\textsc{K.~J. Asciak}, \textsc{M.~A. Francalanza}, \textsc{J.~Lauri}, and \textsc{W.~Myrvold} (2010).
\newblock A survey of some open questions in reconstruction numbers.
\newblock \emph{Ars Combinatoria} \textbf{97}, 443--456.

\bibitem[AMRW96]{arratia1996poisson}
\textsc{Richard Arratia}, \textsc{Daniela Martin}, \textsc{Gesine Reinert}, and \textsc{Michael~S Waterman} (1996).
\newblock Poisson process approximation for sequence repeats, and sequencing by hybridization.
\newblock \emph{Journal of Computational Biology} \textbf{3}(3), 425--463.

\bibitem[BH77]{bondy1977graph}
\textsc{J.~A. Bondy} and \textsc{R.~L. Hemminger} (1977).
\newblock Graph reconstruction---a survey.
\newblock \emph{Journal of Graph Theory} \textbf{1}(3), 227--268.

\bibitem[Bol90]{bollobas1990almost}
\textsc{B\'{e}la Bollob\'{a}s} (1990).
\newblock Almost every graph has reconstruction number three.
\newblock \emph{Journal of Graph Theory} \textbf{14}(1), 1--4.

\bibitem[Bol01]{bollobas2001random}
\textsc{B\'{e}la Bollob\'{a}s} (2001).
\newblock Random graphs, \emph{Cambridge Studies in Advanced Mathematics}, vol.~73.
\newblock Second edn. (Cambridge University Press, Cambridge).

\bibitem[Bon91]{bondy1991graph}
\textsc{J.~A. Bondy} (1991).
\newblock A graph reconstructor's manual.
\newblock \emph{Surveys in combinatorics, 1991 ({G}uildford, 1991)}, \emph{London Math. Soc. Lecture Note Ser.}, vol. 166, 221--252.

\bibitem[DFS94]{dyer1994probability}
\textsc{Martin Dyer}, \textsc{Alan Frieze}, and \textsc{Stephen Suen} (1994).
\newblock The probability of unique solutions of sequencing by hybridization.
\newblock \emph{Journal of Computational Biology} \textbf{1}(2), 105--110.

\bibitem[DYM22]{ding2022shotgun}
\textsc{Jian Ding}, \textsc{Jiangyi Yang}, and \textsc{Heng Ma} (2022).
\newblock Shotgun threshold for sparse {E}rd{\H{o}}s-{R}{\'e}nyi graphs.
\newblock \emph{arXiv preprint arXiv:2208.09876} .

\bibitem[GM22]{gaudio2020shotgun}
\textsc{Julia Gaudio} and \textsc{Elchanan Mossel} (2022).
\newblock Shotgun assembly of {E}rd{\H{o}}s-{R}\'{e}nyi random graphs.
\newblock \emph{Electronic Communications in Probability} \textbf{27}, Paper No. 5, 14.

\bibitem[HT21]{huang2021shotgun}
\textsc{Han Huang} and \textsc{Konstantin Tikhomirov} (2021).
\newblock Shotgun assembly of unlabeled {E}rd{\H{o}}s-{R}{\'e}nyi graphs.
\newblock \emph{arXiv preprint arXiv:2108.09636} .

\bibitem[JKRS23]{johnston2023shotgun}
\textsc{Tom Johnston}, \textsc{Gal Kronenberg}, \textsc{Alexander Roberts}, and \textsc{Alex Scott} (2023).
\newblock Shotgun assembly of random graphs.

\bibitem[Kel42]{kelly1942isometric}
\textsc{Paul~Joseph Kelly} (1942).
\newblock On isometric transformations.
\newblock Ph.D. thesis, University of Wisconsin.

\bibitem[LS16]{lauri2016topics}
\textsc{Josef Lauri} and \textsc{Raffaele Scapellato} (2016).
\newblock Topics in graph automorphisms and reconstruction, \emph{London Mathematical Society Lecture Note Series}, vol. 432.
\newblock Second edn. (Cambridge University Press, Cambridge).

\bibitem[MBT13]{motahari2013information}
\textsc{Abolfazl~S. Motahari}, \textsc{Guy Bresler}, and \textsc{David N.~C. Tse} (2013).
\newblock Information theory of {DNA} shotgun sequencing.
\newblock \emph{Institute of Electrical and Electronics Engineers. Transactions on Information Theory} \textbf{59}(10), 6273--6289.

\bibitem[MR19]{mossel2017shotgun}
\textsc{Elchanan Mossel} and \textsc{Nathan Ross} (2019).
\newblock Shotgun assembly of labeled graphs.
\newblock \emph{IEEE Transactions on Network Science and Engineering} \textbf{6}(2), 145--157.

\bibitem[MS15]{mossel2015shotgun}
\textsc{Elchanan Mossel} and \textsc{Nike Sun} (2015).
\newblock Shotgun assembly of random regular graphs.
\newblock \emph{arXiv preprint arXiv:1512.08473} .

\bibitem[MU17]{mitzenmacher2017probability}
\textsc{Michael Mitzenmacher} and \textsc{Eli Upfal} (2017).
\newblock Probability and computing.
\newblock Second edn. (Cambridge University Press, Cambridge).
\newblock Randomization and probabilistic techniques in algorithms and data analysis.

\bibitem[M{\"u}l76]{muller1976probabilistic}
\textsc{Vladimir M{\"u}ller} (1976).
\newblock Probabilistic reconstruction from subgraphs.
\newblock \emph{Commentationes Mathematicae Universitatis Carolinae} \textbf{17}(4), 709--719.

\bibitem[Rob55]{robbins}
\textsc{Herbert Robbins} (1955).
\newblock A remark on {S}tirling's formula.
\newblock \emph{American Mathematical Monthly} \textbf{62}(1), 26--29.

\bibitem[Ula60]{ulam1960collection}
\textsc{S.~M. Ulam} (1960).
\newblock A collection of mathematical problems.
\newblock Interscience Tracts in Pure and Applied Mathematics, no. 8 (Interscience Publishers, New York-London).

\end{thebibliography}

\end{document}